\documentclass{conm-p-l}
\usepackage[colorlinks,
            linkcolor=reference,
            citecolor=citation,
            urlcolor=e-mail,
            backref]{hyperref}
\usepackage[all]{xy}
\SelectTips{cm}{}
\usepackage[mathscr]{euscript}
\usepackage{tikz}
\usepackage{mathabx}

\usepackage{color}
\definecolor{todo}{rgb}{1,0,0}
\definecolor{conditional}{rgb}{0,1,0}
\definecolor{e-mail}{rgb}{0,.40,.80}
\definecolor{reference}{rgb}{.20,.60,.22}
\definecolor{mrnumber}{rgb}{.80,.40,0}
\definecolor{citation}{rgb}{0,.40,.80}

\newcommand{\ie}{{\em i.e.}\ }
\newcommand{\cf}{{\em cf.}\ }

\newcommand{\ko}{\: , \;}

\renewcommand{\bf}[1]{\mathbf{#1}}

\numberwithin{equation}{subsection}
\newtheorem{theorem}[subsection]{Theorem}
\newtheorem{classification-theorem}[subsection]{Classification Theorem}
\newtheorem{decomposition-theorem}[subsection]{Decomposition Theorem}
\newtheorem{proposition-definition}[subsection]{Proposition-Definition}
\newtheorem{periodicity-conjecture}[subsection]{Periodicity Conjecture}
\newtheorem{lemma}[subsection]{Lemma}
\newtheorem{proposition}[subsection]{Proposition}
\newtheorem{corollary}[subsection]{Corollary}

\newtheorem{remark}[subsection]{Remark}

\newcommand{\reminder}[1]{}

\newcommand{\opname}[1]{\operatorname{\mathsf{#1}}}

\newcommand{\T}{\opname{T}}
\newcommand{\Hasse}{\opname{Hasse}}
\newcommand{\tors}{\opname{tors}}
\newcommand{\ftors}{\opname{ftors}}

\newcommand{\tw}{\opname{tw}\nolimits}

\renewcommand{\mod}{\opname{mod}\nolimits}
\newcommand{\brick}{\opname{brick}\nolimits}
\newcommand{\cjirr}{\opname{cjirr}\nolimits}

\newcommand{\add}{\opname{add}\nolimits}

\newcommand{\im}{\opname{im}\nolimits}

\newcommand{\A}{\mathbb{A}}
\newcommand{\Z}{\mathbb{Z}}
\newcommand{\N}{\mathbb{N}}
\newcommand{\Q}{\mathbb{Q}}

\newcommand{\E}{\mathbb{E}}

\newcommand{\iso}{\xrightarrow{_\sim}}
\newcommand{\liso}{\xleftarrow{_\sim}}

\newcommand{\Hom}{\opname{Hom}}

\newcommand{\Ext}{\opname{Ext}}

\newcommand{\End}{\opname{End}}

\newcommand{\ca}{{\mathscr A}}

\newcommand{\cd}{{\mathscr D}}

\newcommand{\cF}{{\mathscr F}}

\newcommand{\cs}{{\mathscr S}}
\newcommand{\ct}{{\mathscr T}}
\newcommand{\cu}{{\mathscr U}}
\newcommand{\cv}{{\mathscr V}}
\newcommand{\cw}{{\mathscr W}}
\newcommand{\cx}{{\mathscr X}}

\newcommand{\eps}{\varepsilon}
\renewcommand{\phi}{\varphi}

\renewcommand{\hat}[1]{\widehat{#1}}

\newcommand{\bt}{\bullet}

\newcommand{\sgn}{\mbox{sgn}}

\renewcommand{\tilde}[1]{\widetilde{#1}}
\newcommand{\join}{\vee}

\makeatletter
\let\@wraptoccontribs\wraptoccontribs
\makeatother

\begin{document}

\date{June 3, 2018} 

\title{A survey on maximal green sequences}

\author{Bernhard Keller}
\contrib[with an appendix by]{Laurent Demonet}

\address{B. Keller: Universit\'e Paris Diderot -- Paris 7\\
    Sorbonne Universit\'e\\
    UFR de Math\'ematiques\\
    CNRS\\
   Institut de Math\'ematiques de Jussieu--Paris Rive Gauche, IMJ-PRG \\   
    B\^{a}timent Sophie Germain\\
    75205 Paris Cedex 13\\
    France
}
\email{bernhard.keller@imj-prg.fr}
\urladdr{https://webusers.imj-prg.fr/~bernhard.keller/}

\address{L. Demonet: Graduate School of Mathematics\\ 
Nagoya University\\
Chikusa-ku\\
Nagoya\\
464-8602 Japan}
\email{Laurent.Demonet@normalesup.org}
\urladdr{http://www.math.nagoya-u.ac.jp/~demonet/}

\begin{abstract}
Maximal green sequences appear in the study of Fomin--Zelevinsky's cluster algebras.
They are useful for computing refined Donaldson--Thomas invariants,
constructing twist automorphisms and
proving the existence of theta bases and generic bases.
We survey recent progress on their existence and properties
and give a representation-theoretic
proof of Greg Muller's theorem stating that full subquivers inherit maximal green sequences.
In the appendix, Laurent Demonet describes maximal chains of torsion classes
in terms of bricks generalizing a theorem by Igusa.
\end{abstract}




\maketitle

\section{Introduction}
A quiver is an oriented graph. Quiver mutation is an elementary operation
on quivers. It is the basic combinatorial ingredient of Fomin--Zelevinsky's
definition of cluster algebras \cite{FominZelevinsky02}. In this definition,
one recursively constructs generators for the cluster algebra by
repeatedly mutating an initial seed $(Q,x)$ consisting of a quiver $Q$ and a set
$x$ of indeterminates associated with the vertices of $Q$. The
construction process is recorded in a graph, the {\em exchange graph},
whose vertices are the seeds obtained from $(Q,x)$ by iterated
mutation and whose edges correspond to mutations. 
By definition, the edges of the exchange graph are unoriented.
However, it was noticed early on \cite{MarshReinekeZelevinsky03} that
there is a natural partial order on seeds whose minimal inclusions
correspond to edges of the exchange graph, which thus becomes
oriented. For example, for a linearly
oriented quiver of type $A_n$, the poset of seeds is the $n$th Tamari lattice
\cite{IngallsThomas09, Thomas12}.

Maximal green sequences were invented in \cite{Keller10c} 
(and became part of \cite{Keller11c}) but are already
implicit in the work of Gaiotto--Moore--Neitzke \cite{GaiottoMooreNeitzke09}
(published in \cite{GaiottoMooreNeitzke13}). 
A maximal green sequence is a (finite) path in the oriented exchange graph from
the unique smallest element to the unique largest element. Not all quivers
have maximal green sequences but they do exist for important classes of
quivers and
their existence has important consequences: it yields explicit formulas
for Kontsevich--Soibelman's refined Donaldson--Thomas
invariant associated \cite{KontsevichSoibelman08}
with the quiver, for the twist automorphism \cite{GeissLeclercSchroeer12}
of the associated cluster algebra and it is a sufficient condition for
the existence of a theta basis \cite{GrossHackingKeelKontsevich18}
and a generic basis \cite{Qin19} in the upper cluster
algebra. 

In section~\ref{s:mutation},
we review the purely combinatorial definitions of mutation and
green mutation leading to the notion of (maximal) green sequence 
(and, more generally, reddening sequence). In section \ref{s:applications},
we describe the applications of maximal green sequences mentioned above.
We then report on results concerning the existence and non existence
of maximal green sequences (section~\ref{s:existence}). In particular,
we state Greg Muller's theorem to the effect that full subquivers
inherit maximal green sequences. In the final section~\ref{s:proof}, we
give a proof of Muller's theorem based on recent results in the
study of torsion classes \cite{DemonetIyamaReadingReitenThomas17}.
In the appendix, Laurent Demonet establishes a bijection between
maximal chains of torsion classes and maximal forward $\Hom$-orthogonal
sequences of bricks over a finite-dimensional algebra generalizing
a theorem proved by Igusa \cite{Igusa17a} for representation-finite Jacobi algebras.

\section*{Acknowledgments}
I thank Laurent Demonet for his explanations of the results
of \cite{DemonetIyamaReadingReitenThomas17}, his help with
section~\ref{s:proof} and for contributing the appendix. I am grateful to Daniel
Labardini--Fragoso for helpful discussions and for reference \cite{Labardini09a}.
An abridged version of this note was presented at the ICRA 18 in Prague
in August 2018. I thank the organisers for a wonderful conference and for considering
this material for the proceedings. I am indebted to Zheng Hua for pointing
out a misquotation in section~\ref{s:comparing} of a previous version of
this note and to Jiarui Fei for alerting me to reference \cite{Fei16}.
I thank Volker Genz for pointing out an incorrect formulation
of Remark~\ref{rk:support} in the published version of this note.

\section{Mutation and green mutation}
\label{s:mutation}

\subsection{Quiver mutation} A {\em quiver}\index{quiver} is an oriented graph, i.e.~a
quadruple $Q=(Q_0, Q_1, s, t)$ formed by a set of vertices $Q_0$, a set
of arrows $Q_1$ and two maps $s$ and $t$ from $Q_1$ to $Q_0$ which send
an arrow $\alpha$ respectively to its source $s(\alpha)$ and its
target $t(\alpha)$. In practice, a quiver is given by a picture
as in the following example
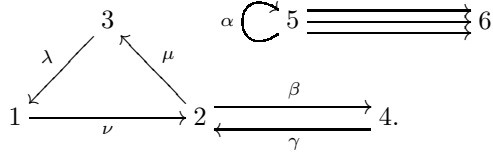
\[ 
\xymatrix{ & 3 \ar[ld]_\lambda & & 5 \ar@(dl,ul)[]^\alpha \ar@<1ex>[rr] \ar[rr] \ar@<-1ex>[rr] & & 6 \\
  1 \ar[rr]_\nu & & 2 \ar@<1ex>[rr]^\beta \ar[ul]_\mu & & 4.
  \ar@<1ex>[ll]^\gamma }
\]
An arrow $\alpha$ whose source and target coincide is a {\em loop}\index{loop}; a
{\em $2$-cycle}\index{$2$-cycle} is a pair of distinct arrows $\beta$ and $\gamma$ such that
$s(\beta)=t(\gamma)$ and $t(\beta)=s(\gamma)$. Similarly, one defines
{\em $n$-cycles} for any positive integer $n$. A vertex $i$ of a quiver
is a {\em source}\index{source of a quiver} (respectively a {\em sink}\index{sink of a quiver}) if there is no arrow
with target $i$ (respectively with source $i$).

By convention, in the sequel, by a quiver, we always mean a finite
quiver without loops nor $2$-cycles whose set of vertices is the
set of integers from $1$ to $n$ for some $n\geq 1$. Up to an
isomorphism fixing the vertices, such a quiver $Q$ is given by
the {\em skew-symmetric matrix $B=B_Q$} whose coefficient $b_{ij}$ is
the difference between the number of arrows from $i$ to $j$ and
the number of arrows from $j$ to $i$ for all $1\leq i,j\leq n$.
Conversely, each skew-symmetric matrix $B$ with integer coefficients
comes from a quiver.

Let $Q$ be a quiver and $k$ a vertex of $Q$. The {\em mutation $\mu_k(Q)$}\index{mutation!of a quiver}
is the quiver obtained from $Q$ as follows:
\begin{itemize}
\item[1)] for each subquiver $\xymatrix{i \ar[r]^\beta & k \ar[r]^\alpha & j}$,
we add a new arrow $[\alpha\beta]: i \to j$;
\item[2)] we reverse all arrows with source or target $k$;
\item[3)] we remove the arrows in a maximal set of pairwise disjoint $2$-cycles.
\end{itemize}
For example, if $k$ is a source or a sink of $Q$, then the
mutation at $k$ simply reverses all the arrows incident with $k$. In general,
if $B$ is the skew-symmetric matrix associated with $Q$ and $B'$ the one
associated with $\mu_k(Q)$, we have
\begin{equation} \label{eq:matrix-mutation}
b'_{ij} = \left\{ \begin{array}{ll}
-b_{ij} & \mbox{if $i=k$ or $j=k$~;} \\
b_{ij}+\sgn(b_{ik})\max(0, b_{ik} b_{kj}) & \mbox{else.}
\end{array} \right.
\end{equation}
This is the {\em matrix mutation rule}\index{mutation!of a matrix} introduced by
Fomin-Zelevinsky in \cite{FominZelevinsky02},
cf. also \cite{FominZelevinsky07}. It applies more generally to 
skew-symmetri\-zable matrices, which correspond to
{\em valued quivers}, cf. section~3.3 of \cite{Keller12a}.

One checks easily that $\mu_k$ is an involution. For example,
the quivers
\begin{equation} \label{quiver1}
\begin{xy} 0;<0.3pt,0pt>:<0pt,-0.3pt>::
(94,0) *+{1} ="0",
(0,156) *+{2} ="1",
(188,156) *+{3} ="2",
"1", {\ar"0"},
"0", {\ar"2"},
"2", {\ar"1"},
\end{xy}
\begin{minipage}{1cm}
\vspace*{1cm}
\begin{center} and \end{center}
\end{minipage}
\begin{xy} 0;<0.3pt,0pt>:<0pt,-0.3pt>::
(92,0) *+{1} ="0",
(0,155) *+{2} ="1",
(188,155) *+{3} ="2",
"0", {\ar"1"},
"2", {\ar"0"},
\end{xy}
\end{equation}
are linked by a mutation at the vertex $1$. Notice that these
quivers are drastically different: The first one is a cycle,
the second one the Hasse diagram of a linearly ordered set.

Two quivers are {\em mutation equivalent}\index{mutation equivalent} if they are linked
by a finite sequence of mutations. For example, it is an easy
exercise to check that any two orientations of a tree are
mutation equivalent. Using the quiver mutation applet \cite{Keller06e}
or the Sage package \cite{MusikerStump11} one can check
that the following three quivers are mutation equivalent
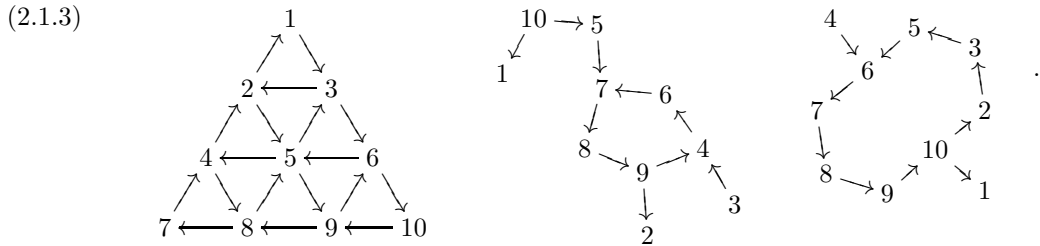
\begin{equation} \label{quiver3}
\begin{xy} 0;<0.6pt,0pt>:<0pt,-0.6pt>::
(79,0) *+{1} ="0",
(52,44) *+{2} ="1",
(105,44) *+{3} ="2",
(26,88) *+{4} ="3",
(79,88) *+{5} ="4",
(131,88) *+{6} ="5",
(0,132) *+{7} ="6",
(52,132) *+{8} ="7",
(105,132) *+{9} ="8",
(157,132) *+{10} ="9",
"1", {\ar"0"},
"0", {\ar"2"},
"2", {\ar"1"},
"3", {\ar"1"},
"1", {\ar"4"},
"4", {\ar"2"},
"2", {\ar"5"},
"4", {\ar"3"},
"6", {\ar"3"},
"3", {\ar"7"},
"5", {\ar"4"},
"7", {\ar"4"},
"4", {\ar"8"},
"8", {\ar"5"},
"5", {\ar"9"},
"7", {\ar"6"},
"8", {\ar"7"},
"9", {\ar"8"},
\end{xy}
\quad\quad
\begin{xy} 0;<0.3pt,0pt>:<0pt,-0.3pt>::
(0,70) *+{1} ="0",
(183,274) *+{2} ="1",
(293,235) *+{3} ="2",
(253,164) *+{4} ="3",
(119,8) *+{5} ="4",
(206,96) *+{6} ="5",
(125,88) *+{7} ="6",
(104,164) *+{8} ="7",
(177,194) *+{9} ="8",
(39,0) *+{10} ="9",
"9", {\ar"0"},
"8", {\ar"1"},
"2", {\ar"3"},
"3", {\ar"5"},
"8", {\ar"3"},
"4", {\ar"6"},
"9", {\ar"4"},
"5", {\ar"6"},
"6", {\ar"7"},
"7", {\ar"8"},
\end{xy}
\quad\quad
\begin{xy} 0;<0.3pt,0pt>:<0pt,-0.3pt>::
(212,217) *+{1} ="0",
(212,116) *+{2} ="1",
(200,36) *+{3} ="2",
(17,0) *+{4} ="3",
(123,11) *+{5} ="4",
(64,66) *+{6} ="5",
(0,116) *+{7} ="6",
(12,196) *+{8} ="7",
(89,221) *+{9} ="8",
(149,166) *+{10} ="9",
"9", {\ar"0"},
"1", {\ar"2"},
"9", {\ar"1"},
"2", {\ar"4"},
"3", {\ar"5"},
"4", {\ar"5"},
"5", {\ar"6"},
"6", {\ar"7"},
"7", {\ar"8"},
"8", {\ar"9"},
\end{xy}
\begin{minipage}{1cm}
\vspace*{1.5cm}
\begin{center} . \end{center}
\end{minipage}
\end{equation}
The common {\em mutation class}\index{mutation class} of these quivers contains 5739 quivers
(up to isomorphism). The mutation class of `most' quivers is infinite.
The classification of the quivers having a finite mutation class
was achieved by Felikson-Shapiro-Tumarkin
\cite{FeliksonShapiroTumarkin12, FeliksonShapiroTumarkin12a}:
in addition to the quivers associated
with triangulations of surfaces (with boundary and marked points,
cf. \cite{FominShapiroThurston08})
the list contains $11$ exceptional quivers, the largest of which is
in the mutation class of the quivers~(\ref{quiver3}).

\subsection{Green quiver mutation}

Let $Q$ be a quiver. The {\em framed quiver $\hat{Q}$} is
obtained from $Q$ by adding, for each vertex $i$, a new vertex
$i'$ and a new arrow $i \to i'$. For example, if $Q$ is the quiver
$1 \to 2$, then the framed quiver $\hat{Q}$ is
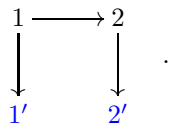
\[
\xymatrix{ 1 \ar[r] \ar[d] & 2 \ar[d] \\ \textcolor{blue}{1'} & \textcolor{blue}{2'}}
\quad\raisebox{-0.5cm}{.}
\]
The new vertices $i'$ are called {\em frozen} vertices, because we never
mutate at them. 
Now suppose that we have transformed $\hat{Q}$ into a quiver $R$ 
by a finite sequence of mutations (at non frozen vertices).
A non frozen vertex $i$ is {\em green} in $R$ if
it is not the target of any arrows ${j\,'} \to i$ from 
frozen vertices $j'$ in $R$. 
It is {\em red} if it is not the source of any arrows $i \to {j\,'}$ to
frozen vertices of $R$, cf. Figure~\ref{fig1}.

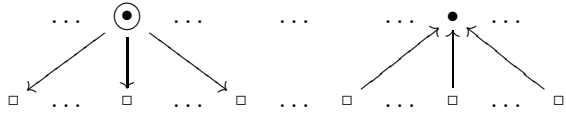
\begin{figure}
\[
\xymatrix@C=0.2cm@R=0.7cm{ & \ldots & \mbox{\textcircled{$\bt$}}
\ar[lld] \ar[d] \ar[rrd] & \ldots 
&  & \ldots & &  \ldots & \bt  & \ldots & \\
\square & \ldots & \square & \ldots & 
\square &  \ldots & \square \ar[urr]  & \ldots & \square \ar[u]  & \ldots & 
\square \ar[ull] 
}
\]
\caption{Green \textcircled{$\bt$} and red $\bt$ vertices in $R$} 
\label{fig1}
\end{figure}

\begin{theorem}[Derksen-Weyman-Zelevinsky \cite{DerksenWeymanZelevinsky10}]
Each non frozen vertex of  $R$ is either green or red.
\end{theorem}

The proof is based on the theory of mutations of quivers with potential
and their decorated representations developed in \cite{DerksenWeymanZelevinsky08,
DerksenWeymanZelevinsky10}.
Alternative proofs of the theorem were given in \cite{Plamondon11a} (via
the cluster category) and in \cite{Nagao13} (via Donaldson--Thomas theory). 
An important 
generalization to valued quivers is proved \cite{GrossHackingKeelKontsevich18}.
As shown in \cite{NakanishiZelevinsky12}, the theorem is central in the theory
of cluster algebras.

The {\em $c$-vector $\alpha_i \in\Z^n$} associated with a non frozen vertex $i$ 
of $R$ is the integer vector whose $j$th component is the difference
between the number of arrows from $i$ to $j'$ minus the number of arrows
from $j'$ to $i$ in $R$. By the theorem, each $c$-vector has either
all components $\geq 0$ or all components $\leq 0$ (sign-coherence of
$c$-vectors). The following definition was first given in \cite{Keller10c, Keller11c}.
A sequence $\bf{i}=(i_1,\ldots, i_N)$ of vertices is
{\em green} if, for each $1\leq t\leq N$, the vertex $i_t$
is green in the partially mutated quiver
\[
\mu_{i_{t-1}} \ldots \mu_{i_2}\mu_{i_1} (\hat{Q}).
\]
It is {\em maximal green} if moreover all non frozen vertices of 
$\mu_{\bf{i}}(\hat{Q})$ are red. It is {\em reddening} \cite{Keller17}
(or {\em green-to-red} \cite{Muller16})
if it is not necessarily green but all non frozen vertices of $\mu_{\bf{i}}(\hat{Q})$ 
are red.

In Figure~\ref{fig2}, we see that the quiver $\vec{A}_2$ has exactly
two maximal green sequences: $(1,2)$ and $(2,1,2)$ (green vertices
are encircled). The final quivers in the two sequences are
isomorphic by a {\em frozen isomorphism} (i.e. an isomorphism which
fixes the frozen vertices) to the {\em coframed quiver $\widecheck{Q}$}, which is obtained
from $Q$ by adding, for each vertex $i$, a new vertex $i'$ and an
arrow $i'\to i$. This is a general phenomenon:

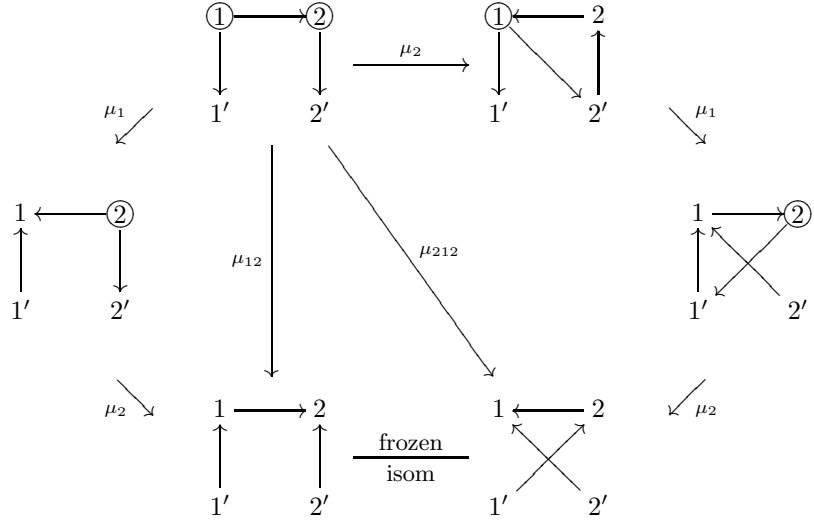
\begin{figure}
\[
\def\g#1{\save [].[dr]!C *++\frm{}="g#1" \restore}
\xymatrix{
 & &  *+[o][F-]{1}\g1 \ar[r] \ar[d] & *+[o][F-]{2} \ar[d] & & *+[o][F-]{1}\g2  \ar[d] \ar[dr] & 2 \ar[l] & &  \\
 & &   1'                & 2' & &    1'               & 2' \ar[u]\\
\g3 1 & *+[o][F-]{2} \ar[l] \ar[d]  &  &  &   &    &   & \g4 1 \ar[r] & *+[o][F-]{2} \ar[dl] \\
1'\ar[u] & 2'            &  &  &   &    &   &  1' \ar[u] & 2' \ar[ul] \\
 & & \g5 1 \ar[r] & 2         &  &  \g6 1 & 2 \ar[l] \\
 & &    1' \ar[u] & 2' \ar[u] &  &   1' \ar[ur] & 2' \ar[ul]
 \ar@{->}^{\mu_2} "g1"; "g2"
 \ar@<-20pt>@{->}_{\mu_1} "g1"; "g3"
 \ar@<20pt>@{->}^{\mu_1} "g2"; "g4"
 \ar@<-20pt>@{->}_{\mu_2} "g3"; "g5"
 \ar@<20pt>@{->}^{\mu_2} "g4"; "g6"
 \ar@{-}^{\mbox{\small frozen}}_{\mbox{\small isom}} "g5"; "g6"
 \ar_{\mu_{12}} "g1"; "g5"
 \ar^{\mu_{212}} "g1"; "g6"
 }
\]
\caption{The two maximal green sequences for $\vec{A}_2$}
\label{fig2}
\end{figure}

\begin{proposition}[Prop.~2.10 of \cite{BruestleDupontPerotin14}]
\label{prop:permutation}
Suppose that $Q$ admits a reddening sequence $\bf{i}$. Then there is
a unique isomorphism $\mu_{\bf{i}}(\hat{Q}) \iso \widecheck{Q}$ fixing
the frozen vertices and sending a non frozen vertex $i$ to
$\sigma(i)$ for a unique permutation $\sigma$ of the vertices of $Q$.
\end{proposition}

For maximal green sequences of quivers in the mutation class of $A_n$, 
the permutation $\sigma$ is studied in  
\cite{GarverMusiker17, IgusaZhou16}. In general,
it remains mysterious.

It is a fact that large classes of quivers appearing in Lie theory and
higher Teichm\"uller theory do admit maximal green sequences. We refer
to section~\ref{s:existence} below for classes of examples. This is interesting
because of the applications sketched in the next section.

\section{Applications}
\label{s:applications}

\subsection{Refined Donaldson--Thomas invariants}
The \emph{quantum dilogarithm series} is defined by
\begin{align*}
\E(y)  &\ = 1 + \frac{q^{1/2}}{q-1}\cdot y + \cdots  +
\frac{q^{n^{\,2}/2} y^n}{(q^n-1)(q^n-q) \cdots (q^n-q^{n-1})} + \cdots \\
       &\ \in \Q(q^{1/2})[[y]] ,
\end{align*}
where $q^{1/2}$ is an indeterminate whose square is denoted by $q$ and
$y$ is an indeterminate. This series is a classical object with many
remarkable properties, cf. for example \cite{Zagier91}. 

Let $Q$ be a quiver with $n$ vertices.
Let $\lambda_Q:\Z^n \times \Z^n \to \Z$ be 
the bilinear antisymmetric form associated with the matrix $B_Q$.
Define the {\em complete quantum affine space} as the algebra
\[
\hat{\A}_Q = \Q(q^{1/2})\langle\!\langle y^\alpha, \alpha\in\N^n \;|\;
y^\alpha y^\beta = q^{1/2 \,\lambda(\alpha,\beta)} y^{\alpha+\beta} \rangle\!\rangle.
\]
This is a slightly non commutative deformation of an ordinary commutative power
series algebra in the $n$ indeterminates $y_i=y^{e_i}$, where $e_i$ is the
$i$th vector of the standard basis of $\Z^n$. We write $\A_Q$ for the non
completed variant of $\hat{\A}_Q$.
For a sequence $\bf{i}=(i_1, \ldots, i_N)$ of vertices of $Q$,
we define
\[
\E_{Q,\bf{i}} = \E(y^{\eps_1 \beta_1})^{\eps_1} \cdots \E(y^{\eps_N \beta_N})^{\eps_N} ,
\]
where the product is taken in $\hat{\A}_Q$, the vector $\beta_t$ is
the $c$-vector associated with the vertex $i_t$ of the partially
mutated quiver
\[
\mu_{i_{t-1}} \ldots \mu_{i_1}(\hat{Q})
\]
and $\eps_i$ is the common sign of its entries, $1 \leq t \leq N$. Notice that
by the sign-coherence, each factor does belong to $\hat{\A}_Q$.

Let $RDT_Q\in \hat{\A}_Q$ denote the refined Donaldson--Thomas invariant of
$Q$ constructed by Kontse\-vich--Soibelman 
\cite{KontsevichSoibelman08, KontsevichSoibelman10a, KontsevichSoibelman11}.
Notice that their construction has not yet been made completely rigorous due
to technical difficulties arising from the fact that the potentials needed are
{\em infinite} linear combinations of cycles.
The following theorem was the motivation for \cite{Keller10c}. Independently,
it was discovered in the study of the BPS spectrum
by Gaiotto--Moore--Neitzke in \cite{GaiottoMooreNeitzke09}
and used in the physics literature, for example in 
\cite{CecottiEtAl14, CecottiCordovaVafa11}, 
cf. \cite{Xie16} and the references given there.

\begin{theorem} \label{thm:identity} If $Q$ admits a reddening sequence $\bf{i}$, we
have 
\[
RDT_Q= \E_{Q,\bf{i}} \quad\mbox{in}\quad \hat{\A}_Q.
\]
\end{theorem}

The theorem is proved in section~7 of \cite{Keller12} under the assumption
that $RDT_Q$ is well-defined. It implies in particular that the right hand
side is independent of the choice of $\bf{i}$. This was conjectured in
\cite{Nagao11} and can be proved
rigorously using the theory of cluster algebras and their additive
categorification, cf. section~7.11 of \cite{Keller12}. 

Thanks to the theorem, each pair $(\bf{i},\bf{i'})$ of reddening sequences
yields a quantum dilogarithm identity. For example, the 
two maximal green sequences of the quiver $Q=\vec{A}_2$ yield
the {\em pentagon identity}
\begin{equation} \label{eq:pentagon}
\E(y_1) \E(y_2) =\E(y_2) \E(q^{-1/2} y_1 y_2) \E(y_1)
\end{equation}
due to \cite{FaddeevVolkov93} and \cite{FaddeevKashaev94}, 
cf. \cite{Volkov12} for a recent account. Analogous pairs of
maximal green sequences exist for all Dynkin quivers and 
yield generalizations of the pentagon identity
due to Reineke \cite{Reineke10}, cf. \cite{Keller11c}.
These identities further generalize to square products \cite{Keller13} of
Dynkin quivers, cf. \cite{Keller11c}. For pairs $(A_n,A_m)$, an alternative geometric
proof of the identities is given by Allman--Rimanyi \cite{AllmanRimanyi18}.
We refer to section~5 of \cite{Keller17} for more examples.

One may ask whether all the quantum dilogarithm identities obtained from
the theorem are in fact consequences of the pentagon identity. This
is not the case. Counterexamples based on \cite{FominShapiroThurston08}
can be found in \cite{KimYamazaki20}. However, it does hold
if $Q$ is an affine ayclic quiver, as shown 
by Hermes--Igusa \cite{HermesIgusa19}.

In \cite{GoncharovShen18}, Goncharov--Shen construct maximal green sequences
for a large class of quivers appearing in higher Teichm\"uller theory and
apply the theorem to obtain the corresponding Donaldson--Thomas
invariants. Similar results for Grassmannians and double Bruhat cells
are due to Weng \cite{Weng16a, Weng16b, Weng16c} and for
double Bott--Samelson cells to Shen--Weng \cite{ShenWeng19}.

\subsection{Twist automorphisms} Let $Q$ be a quiver admitting a
reddening sequence $\bf{i}$. Let $\sigma$ be the permutation of
Prop.~\ref{prop:permutation}. Then the composition of the mutation
sequence $\mu_{\bf{i}}$ with the permutation $\sigma$ transforms
the initial seed $(Q,x)$ of the cluster algebra $\ca_Q$ into a seed
of the form $(Q,u)$ and thus yields an automorphism
\[
\tw: \ca_Q \iso \ca_Q \ko x_i \mapsto u_i
\]
called the {\em twist automorphism} of $\ca_Q$. 

From the independence of $\E_{Q,\bf{i}}$ of the choice of the reddening
sequence $\bf{i}$, one can deduce that the twist automorphism
$\tw$ is also independent of $\bf{i}$ (\cf section~4 of \cite{Keller11c}
and section~6.4 of \cite{Keller12a}).

As shown by Geiss--Leclerc--Schr\"oer \cite{GeissLeclercSchroeer12},
for unipotent cells of Kac--Moody groups, the
twist automorphism identifies with the chamber ansatz of
Berenstein--Fomin--Zelevinsky \cite{BerensteinFominZelevinsky96}.
It has found important applications in the work of
Marsh--Scott \cite{MarshScott16}, Muller--Speyer \cite{MullerSpeyer17},
Rietsch--Williams \cite{RietschWilliams17},
Cautis--Williams \cite{CautisWilliams18}, \ldots.

\subsection{The Fock--Goncharov conjectures}

Let $Q$ be a valued quiver. The Langlands dual valued quiver $Q^L$ is
obtained by reversing all the valuations of $Q$ (without changing the arrows).
The skew-symmetrizable matrix $B_{Q^L}$ is the opposite transpose 
$-B_Q^T$ of $B_Q$.
Thus, we have $Q=Q^L$ if $Q$ is an ordinary (i.e. equivalued) quiver.
Suppose that $Q$ is obtained from an ice quiver $\tilde{Q}$ by removing
all the frozen vertices and all the arrows incident with them. Assume that
the exchange matrix associated with 
$\tilde{Q}$ is of maximal rank.

\begin{theorem}[Gross--Hacking--Keel--Kontsevich 
\cite{GrossHackingKeelKontsevich18}] If $Q$ has a reddening sequence,
then the Fock--Goncharov duality conjectures  \cite{FockGoncharov09}
hold for $Q$ and in particular the upper cluster algebra
$\cu_{\tilde{Q}}$ admits a basis parametrized by the tropical points
of the cluster Poisson variety associated with $\tilde{Q}^L$.
\end{theorem}

The basis constructed by Gross--Hacking--Keel--Kontsevich is known
as the {\em theta basis}. 
{\em Generic bases} were first considered by Dupont \cite{Dupont11} for acyclic
quivers. They are constructed using generic values of 
cluster characters \cite{CalderoChapoton06, Plamondon16}.
In \cite{GeissLeclercSchroeer12}, Geiss--Leclerc--Schr\"oer showed
that they exist for large classes of cluster algebras arising in Lie theory
and coincide with Lusztig's dual semi-canonical bases. Plamondon
showed in \cite{Plamondon13} that generic bases are also canonically
parametrized by the tropical points of the cluster Poisson variety.

\begin{theorem}[Qin \cite{Qin19}] Let $Q$ be an (equivalued) quiver.
 If $Q$ has a reddening sequence,
then the upper cluster algebra $\cu_{\tilde{Q}}$ admits a generic basis parametrized 
by the tropical points of the cluster Poisson variety associated with $\tilde{Q}$.
\end{theorem}

Notice that these results concern the upper cluster algebra.
It is expected that the existence of a maximal green, or reddening, sequence
should have implications for the relationship between the cluster algebra
and the upper cluster algebra. Notice however that this relationship depends
on the choice of coefficients 
\cite[Example 8.3]{Fei16} \cite{BucherMachacekShapiro18}, whereas
the oriented exchange graph, and thus the existence of a maximal
green sequence, does not \cite{CerulliKellerLabardiniPlamondon13}.
A conjecture on the precise relationship is formulated in
\cite{Mills18a}.

\section{Existence and properties}
\label{s:existence}

\subsection{Existence} Suppose that $Q$ is an acyclic quiver.
A {\em source sequence} for $Q$ is an enumeration of the vertices
of $Q$ which is increasing for the partial order defined by the
existence of a path. It is easy to check that each source sequence
is a maximal green sequence for $Q$, cf.~\cite{BruestleDupontPerotin14}.

Recall that a valued quiver is cluster-finite (i.e. the associated cluster algebra
has only finitely many cluster variables) iff it is mutation-equivalent
to an orientation of a Dynkin diagram \cite{FominZelevinsky03}. As observed in
in \cite{BruestleDupontPerotin14}, it is immediate from
\cite{ChapotonFominZelevinsky02} that cluster-finite valued quivers have
maximal green sequences.

For valued quivers mutation-equivalent to orientations of extended
Dynkin diagrams, there are only finitely many maximal green
sequences \cite{BruestleHermesIgusaTodorov17}. The same holds for 
(equally valued) acyclic quivers \cite{IgusaZhou19}.

Two canonical maximal green sequences exist for the square
product of two alternating valued Dynkin quivers, cf.~section~5 of \cite{Keller17}.

If $R$ is an ayclic quiver and $\tilde{w}$ a reduced expression for an element
of the Coxeter group associated with the underlying graph of $R$, there is
a canonical quiver $Q$ associated with the pair $(R, \tilde{w})$. It serves
to obtain a cluster structure on the coordinate algebra of the unipotent cell
associated with $w$ in the Kac--Moody group determined by $R$,
cf.~\cite{BuanIyamaReitenScott09} and
\cite{GeissLeclercSchroeer11b}. It is closely related to the 
(upper) cluster structures on Bruhat cells obtained in
\cite{BerensteinFominZelevinsky05}.
In section~13 of \cite{GeissLeclercSchroeer11b}, Geiss--Leclerc--Schr\"oer
exhibit a canonical reddening sequence for $Q$ (we conjecture that it is
actually maximal green). 

In analogy with the definition of a full subcategory, one defines a 
{\em full subquiver $Q'$} of a quiver $Q$ to be a subquiver containing all
the arrows in $Q$ between any two of its vertices.

\begin{theorem}[Muller \cite{Muller16}] \label{thm:Muller}
If $Q$ has a maximal green 
(resp. reddening) sequence,
then each full subquiver $Q'\subseteq Q$ has a maximal green 
(resp. reddening) sequence.
\end{theorem}

\begin{remark} \label{rk:support}
Muller's proof shows that more precisely, starting from a maximal green sequence
$(k_1, \ldots, k_N)$  for $Q$, one obtains a maximal green sequence for $Q'$ 
as follows: Let $(c_1, \ldots, c_N)$ be the sequence of $c$-vectors associated
with the given sequence $(k_1, \ldots, k_N)$. Form the subsequence
$(c'_1, \ldots, c'_M)$ of $(c_1, \ldots, c_N)$ formed by those vectors
supported on $Q'$. Then there is a unique sequence of mutations
$(k'_1, \ldots, k'_M)$ of $Q'$ whose associated sequence of $c$-vectors
is $(c'_1, \ldots, c'_M)$. The sequence $(k'_1, \ldots, k'_M)$ is the
required maximal green sequence for $Q'$.
\end{remark}

Muller's proof uses the existence and uniqueness of the scattering diagram \cite{KontsevichSoibelman06a,
GrossSiebert11} associated with a quiver \cite{GrossHackingKeelKontsevich18}. 
We will give a representation-theoretic proof in section~\ref{s:proof}.

Let $Q$ be a quiver and $Q'$, $Q''$ full subquivers. We say that
$Q$ is a {\em triangular extension} of $Q'$ by $Q''$ if the set of vertices
of $Q$ is the disjoint union of the sets of vertices of $Q'$ and $Q''$ and
there are no arrows from vertices of $Q''$ to vertices of $Q'$.
After pioneering work in 
\cite{GarverMusiker17}, the following theorem was proved in
\cite{CaoLi17} using Muller's theorem \ref{thm:Muller}.

\begin{theorem} If $Q$ is a triangular extension of $Q'$ by $Q''$, then
$Q$ has a maximal green sequence if and only if $Q'$ and $Q''$ have
maximal green sequences.
\end{theorem}

We refer to \cite{BucherMachacek18, BucherMachacekRunburgYeckZewde19}
for recent extensions of this theorem. The case of mutation-finite quivers
is treated in section~\ref{ss:mutation-finite} below.

\subsection{Preservation under mutations} We have the following
`rotation lemma'.

\begin{lemma}[Br\"ustle--Hermes--Igusa--Todorov \cite{BruestleHermesIgusaTodorov17}]
If $\bf{i}=(i_1, \ldots, i_N)$ is maximal green (resp. reddening) for a quiver $Q$,
then $(i_2, \ldots, i_N, k)$ is maximal green (resp. reddening)
for $\mu_{i_1}(Q)$, where $k$ is
the target of the unique arrow with source $i_1'$ in $\mu_{\bf{i}}(\hat{Q})$.
\end{lemma}

\begin{theorem}[Muller \cite{Muller16}] 
If $Q$ admits a reddening sequence,
then each quiver mutation-equivalent to $Q$ admits a reddening sequence.
\end{theorem}

However, the analogous statement for maximal green sequences is false,
as we will see below.

\subsection{Non existence} For three non negative integers $a$, $b$, $c$, denote
by $Q_{a,b,c}$ the quiver with three vertices $1$, $2$, $3$ and $a$ arrows from
$1$ to $2$, $b$ arrows from $2$ to $3$ and $c$ arrows from $3$ to $1$
(cf. \cite{Seven14} for a study of the case of valued quivers with
$3$ vertices).
As shown in \cite{BruestleDupontPerotin14}, the quiver $Q_{2,2,2}$ does
not admit a maximal green sequence (nor does it admit a reddening sequence).
Muller shows \cite{Muller16} that none of the quivers $Q_{a,b,c}$ with
all three numbers $a$, $b$, $c\geq 2$ admits a maximal green sequence
On the other hand, the quiver $Q_{2,2,3}$ is mutation-acyclic
(the quiver $Q_{a,b,c}$ is mutation-acyclic iff $\min(a,b,c)<2$ or
$a^2+b^2+c^2-abc>4$ as shown in 
\cite{BeinekeBruestleHille11}). Thus, the existence of a maximal green
sequence is not preserved under mutation \cite{Muller16}.

\subsection{Existence and non existence for mutation-finite quivers}
\label{ss:mutation-finite}
Generalizing the example of $Q_{2,2,2}$, Ladkani has shown in
\cite{Ladkani13} that the quivers associated \cite{FominShapiroThurston08}
with once-punctured
surfaces of arbitrary genus (without boundary) do not admit
reddening sequences (his proof is based on the work of Labardini--Fragoso
\cite{Labardini09a} and Corollary~\ref{cor:finite-dimensional} below).
Seven shows in \cite{Seven14a} that the quivers in the mutation class
of the quiver $X_7$ (discovered in \cite{DerksenOwen08}) 
do not admit maximal green sequences (and presumably
no reddening sequences either).

Partial results in the direction of the following theorem
were obtained in \cite{CecottiEtAl14, Bucher14, BucherMills18,CormierDilleryReshSerhiyenko16, GarverMusiker17}.

\begin{figure}
\[
\begin{xy} 0;<0.4pt,0pt>:<0pt,-0.4pt>:: 
(50,0) *+{1} ="0",
(150,0) *+{2} ="1",
(200,75) *+{3} ="2",
(200,200) *+{4} ="3",
(0,75) *+{5} ="4",
(0,200) *+{6} ="5",
(100,125) *+{7} ="6",
"0", {\ar|*+{\scriptstyle 2}"1"},
"6", {\ar"0"},
"1", {\ar"6"},
"2", {\ar|*+{\scriptstyle 2}"3"},
"6", {\ar"2"},
"3", {\ar"6"},
"4", {\ar|*+{\scriptstyle 2}"5"},
"6", {\ar"4"},
"5", {\ar"6"},
\end{xy}
\]
\caption{The quiver $X_7$ (double arrows are marked with $2$)}
\end{figure}
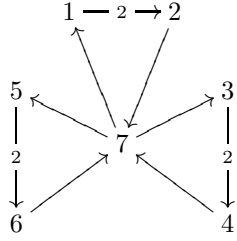

\begin{theorem}[Mills \cite{Mills17}]
If $Q$ is a mutation-finite quiver, it has a maximal green sequence except 
if it comes from a once-punctured closed surface of genus $\geq 1$ 
or is in the mutation class of $X_7$.
\end{theorem}

Information on the length of the minimal length maximal green sequences
for the quivers associated with annuli or punctured disks can be found
in \cite{Kase15, GarverMcConvilleSerhiyenko17, GarverMcConvilleSerhiyenko18}.
Maximal green sequences for minimal mutation-infinite quivers are
studied in \cite{LawsonMills18}.

\section{Muller's theorem via representation theory}
\label{s:proof}

Let $Q$ be a quiver admitting a maximal green sequence and $Q'\subseteq Q$
a full subquiver. We wish to show that $Q'$ has a maximal green sequence as well.

We recall the setup of Derksen--Weyman--Zelevinsky's theory
of quivers with potentials and their mutations \cite{DerksenWeymanZelevinsky08}.
Let $k$ be an uncountable field. Let $kQ$ be the path algebra and
$J$ the two-sided ideal of $kQ$ generated by the arrows. Let
$\hat{kQ}$ be the completed path algebra, i.e. inverse limit
of the finite-dimensional quotients $kQ/J^n$, $n\geq 1$. Let
$W$ be a non degenerate potential, i.e. an element of $\hat{kQ}$
which is an infinite linear combination of cycles of length $\geq 3$
such that the pair $(Q,W)$ can be mutated indefinitely without
creating $2$-cycles in the quiver component
$Q'$ of the mutated quiver with potential $(Q',W')$. 
Let $A$ be the Jacobi algebra of $(Q,W)$,
i.e. the quotient of $\hat{kQ}$ by the closed ideal generated
by the cyclic derivatives $\partial_\alpha W$, $\alpha\in Q_1$.
Notice that in general, the algebra $A$ is infinite-dimensional.
Let $\mod A$ be the category of finite-dimensional (right) $A$-modules.
Recall that all finite-dimensional $A$-modules are nilpotent, i.e.
annihilated by a sufficiently high power of the ideal $J$
(cf. section~10 of \cite{DerksenWeymanZelevinsky08}). Therefore,
the simple objects of $\mod A$ are the simple modules
$S_i$ associated with the vertices $i\in Q_0$. Thus,
the abelian category $\mod A$ is an $\Ext$-finite length category
with finitely many simple objects, cf. \cite{Gabriel73a, KrauseVossieck17}.

Let $\ca$ be an abelian category. For a
class of objects $\cx$, we denote by $\cx^\perp$ the {\em right
orthogonal of $\cx$}, i.e. the full subcategory of objects $Y$
such that $\Hom(X,Y)=0$ for all $X\in\cx$. We denote by
$\add(\cx)$ the full subcategory formed by all direct summands
of finite direct sums of objects of $\cx$.
Recall that a {\em torsion pair} \cite{Dickson66} in $\ca$ is a pair of full 
subcategories $(\ct,\cF)$ such
that $\Hom(T,F)=0$ for all $T\in\ct$ and $F\in \cF$ and for each object
$M$ of $\ca$, there is a short exact sequence
\[
\xymatrix{0 \ar[r] & T \ar[r] & M \ar[r] & F \ar[r] & 0}
\]
with $T\in\ct$ and $F\in\cF$. In this case, $\ct$ is called a {\em torsion class}.
The torsion classes of the category of finite-dimensional modules over
a finite-dimensional algebra are precisely the full subcategories closed
under extensions and quotients \cite{Dickson66}.  A {\em brick} is 
an object whose endomorphism algebra is a
division algebra.

The link between torsion classes and clusters is already implicit
in Marsh--Reineke--Zelevinsky's
\cite{MarshReinekeZelevinsky03}. It is made completely
explicit by Ingalls--Thomas in \cite{IngallsThomas09} (in the acyclic case).
As explained in section~7.6 of \cite{Keller12a}, the following
theorem results from Nagao's work in \cite{Nagao13}.
For the case of acyclic quivers, cf. \cite{Qiu15}.

\begin{theorem}[Nagao] Each green sequence
$\bf{i}=(i_1, \ldots, i_N)$ yields an ascending chain 
\[
0 = \ct_0 \subset \ct_1 \subset \ct_2 \subset \ldots \subset \ct_N \subseteq \mod A
\]
of torsion classes such that
\begin{equation}
\label{eq:min-inclusion}
\ct_{t-1}^\perp \cap \ct_t = \add(B_t)
\end{equation}
for a unique brick $B_t$ with $\End(B_t)=k$, $\Ext^1(B_t,B_t)=0$ and such
that the dimension vector of $B_t$ is the $c$-vector associated with the
vertex $i_t$, $1 \leq t\leq N$.
Moreover, the green sequence $\bf{i}$ is maximal if and only if
$\ct_N=\mod A$.
\end{theorem}

Notice that the theorem admits an obvious
generalization to arbitrary (red and green) sequences of mutations 
(cf. section~7.6 of \cite{Keller12a}). We refer to \cite{GarverMcConville19} 
for applications of the theorem to the study of lattice properties
of oriented exchange graphs.

\begin{corollary}[Br\"ustle--Dupont--P\'erotin \cite{BruestleDupontPerotin14}]
\label{cor:finite-dimensional}
If $Q$ admits a reddening (in particular a maximal green) sequence, 
then $A$ is finite-dimensional.
\end{corollary}

\begin{proof} By the theorem, each object of
$\mod A$ admits a finite zig-zag-filtration, whose subquotients are finite direct
sums of the finitely many modules $B_t$, $1\leq t\leq N$. So in this case, there
is a uniform bound on the Loewy length of the finite-dimensional $A$-modules,
which implies that $A$ itself is finite-dimensional 
\cite{Gabriel73, KrauseVossieck17}, a fact proved in a different manner in
\cite{BruestleDupontPerotin14}. 
\end{proof}

Recall that the Hasse quiver $\Hasse(P)$ of a poset $P$ has as vertices the
elements of $P$ and an arrow $x\to y$ for each minimal inequality
$x<y$, i.e. we have $x<y$ and whenever $x\leq z\leq y$, we have
$x=z$ or $z=y$. It is immediate from the equality~(\ref{eq:min-inclusion})
that the inclusions $\ct_{t-1}\subset \ct_t$ in the chain of torsion
classes associated to a green sequence are minimal inclusions.
Thus, the chain yields a path starting at $0$ in the Hasse
quiver $\Hasse(\tors A)$ of the poset of torsion classes
$\tors A$ of $\mod A$. The following
theorem is an immediate consequence of the results of
Adachi--Iyama--Reiten \cite{AdachiIyamaReiten14} and
Demonet--Iyama--Jasso \cite{DemonetIyamaJasso19}.

\begin{theorem} \label{thm:bijection}
Suppose $\dim A<\infty$. Then Nagao's map taking a 
green sequence $\bf{i}$ to the chain of torsion classes $(\ct_t)$ is
a bijection from the set of green sequences for $Q$ to the set
of paths starting at $0$ in the Hasse quiver of torsion classes
of $\mod A$.
\end{theorem}

\begin{proof} Recall that a torsion class $\ct$ is functorially finite if
for each module $M$, there is a morphism $M \to T$ with $T\in\ct$ such
that each morphism $M \to T'$ with $T'\in \ct$ factors through $M \to T$.
Let us write $\ftors A$ for the poset of functorially finite torsion
classes in $\mod A$. It follows from
Theorem~4.1 of \cite{AdachiIyamaReiten14}, cf. also
Theorem~4.9 of \cite{BruestleYang13}, that Nagao's map
is a bijection from the set of green sequences onto the set of
paths starting at $0$ in the quiver $\Hasse(\ftors A)$. It follows from
Theorem~1.3 of \cite{DemonetIyamaJasso19} that immediate
successors and predecessors in $\Hasse(\tors A)$ of functorially finite torsion
classes are functorially finite. Thus the inclusion 
$\Hasse(\ftors A) \subset \Hasse(\tors A)$ induces an isomorphism
of the connected component of $0$ in $\Hasse(\ftors A)$ onto the
connected component of $0$ in $\Hasse(\tors A)$. The claim follows.
\end{proof}

Though it is not necessary for the proof of Muller's theorem,
it is an interesting question to ask which sequences of bricks $(B_t)$ are
associated with maximal green sequences of $Q$. The following
remarkably simple criterion is proved for maximal green sequences
of cluster-finite quivers by Igusa
in Corollary~2.14 of \cite{Igusa17a}. A more general
statement concerning (possibly infinite) chains in the poset of torsion
classes of a finite-dimensional algebra is proved by
Demonet in Appendix A. Applications to the construction of maximal
green sequences for representation-finite cluster-tilted algebras 
are given in the appendix to \cite{Igusa17a} for type $A$ and in
\cite{NasrIsfahani18} for arbitrary type.

\begin{theorem}[Igusa \cite{Igusa17a}, Demonet App.~A]
Suppose $\dim A<\infty$. A sequence of
bricks $B_1, \ldots, B_N$ is associated with a maximal green sequence for $Q$
if and only if
\[
\Hom(B_i, B_j)=0 \mbox{ for all } i<j
\]
and the sequence cannot be refined keeping this condition.
\end{theorem}

To conclude the proof of Muller's theorem, we need one more
recent result on torsion classes. By definition, if $R$ and $R'$ are
quivers, a {\em contraction} $R \to R'$ is a functor from
the path category of $R$ to that of $R'$ which maps each
arrow to an arrow or an identity morphism.

\begin{theorem}[Demonet--Iyama--Reading--Reiten--Thomas 
\cite{DemonetIyamaReadingReitenThomas17}]
Let $A$ be a finite-dimensional algebra and $I\triangleleft A$ a $2$-sided ideal.
Then the map $\ct \mapsto \ct \cap \mod(A/I)$
induces a contraction
\[
\Hasse(\tors A) \to \Hasse(\tors A/I).
\]
\end{theorem}

\begin{proof} Let us first recall from Theorem~3.3 of 
\cite{DemonetIyamaReadingReitenThomas17},
\cf also Theorem~1.0.2 of \cite{BarnardCarrollZhu17}, that an inclusion of
torsion classes $\cs\subseteq\ct$ is minimal if and only if
there is a unique (up to isomorphism) brick in $\cs^\perp\cap\ct$ and
that it is an equality if and only if there is no brick in $\cs^\perp\cap\ct$.
Now let $\cs\subset\ct$ be a minimal inclusion of torsion classes
in $\mod A$. Let $\cs_I=\cs\cap\mod(A/I)$ and $\ct_I=\ct\cap\mod(A/I)$.
Let $M$ be a module in $\cs_I^\perp\cap\ct_I$. Let $M_\cs\subseteq M$
be its maximal submodule in $\cs$. Then $M_\cs$ clearly belongs
to $\cs_I$. Since $M$ is right orthogonal to $\cs_I$, we have $M_\cs=0$.
Thus, the module $M$ belongs to $\cs^\perp\cap\ct$ and we have
$\cs_I^\perp\cap\ct_I \subset \cs^\perp\cap\ct$. Thus, the subcategory
$\cs_I^\perp\cap\ct_I$ contains either zero or one brick and
the inclusion $\cs_I\subseteq\cs_I$ is either an equality or minimal.
\end{proof}

\begin{remark} \label{rk:bricks}
As explained in section~3.2 of 
\cite{DemonetIyamaReadingReitenThomas17}, a brick $B$ is associated
with each minimal inclusion of torsion classes $\cs\subset\ct$ of $\mod A$
and the proof shows that such an inclusion is mapped to
a minimal inclusion in $\mod(A/I)$ if and only if $B$ belongs to $\mod(A/I)$
(i.e. $B$ is annihilated by $I$). 
\end{remark}

We can now conclude: Suppose that $Q$ has a maximal green sequence
and that $Q'\subseteq Q$ is a full subquiver. Let $e$ be the sum
of the lazy idempotents $e_i$ associated with the vertices not in $Q'_0$.
Clearly, the algebra $\hat{kQ'}$ is isomorphic to the quotient of
$\hat{kQ}$ by the two-sided ideal generated by $e$.
Let $W' \in \hat{kQ'}$ be the image of $W$ under the projection
$\hat{kQ} \to \hat{kQ'}$. By Corollary~22 of \cite{Labardini09a}
(the published version!), the potential $W'$ is non degenerate on $Q'$. 
Let $A'$ be the Jacobi algebra of $(Q',W')$. Clearly, the algebra $A'$
is isomorphic to the quotient of $A$ by the two-sided ideal generated by $e$. 
By the theorem above, the map
\[
\Hasse(\tors A) \to \Hasse(\tors A') \ko \ct \mapsto \ct\cap\mod A'
\]
is a contraction and clearly, it takes $0$ to $0$ and $\mod A$ to
$\mod A'$. Thus, it maps a finite path from $0$ to $\mod A$ to
a finite path from $0$ to $\mod A'$. By Theorem~\ref{thm:bijection},
we obtain a maximal green sequence for $Q'$.

\begin{remark} By comparing remarks~\ref{rk:support} and \ref{rk:bricks} we
see that both proofs yield the same explicit recipe for constructing
the induced maximal green sequence.
\end{remark}

\section{Comparing the proofs}
\label{s:comparing}

Let us emphasize that the statement of Muller's theorem~\ref{thm:Muller}
and the proof via scattering diagrams
given by him in \cite{Muller16} go through for valued
quivers. In contrast, our representation-theoretic proof only works
for the classes of valued quivers treated by Demonet
\cite{Demonet10, Demonet11}.

When trying to compare the two proofs we are naturally lead to the
problem of relating torsion classes to scattering diagrams. A first step
towards its solution was taken by Bridgeland \cite{Bridgeland17}
who associates a scattering diagram (with values in a motivic Hall
algebra) to each finite quiver with (polynomial) relations and 
without loops or $2$-cycles. 
By a result of Hua--Song \cite{HuaSong18}, each potential on a quiver
$Q$ whose (complete) Jacobi-algebra is finite-dimensional 
is right-equivalent to a polynomial potential $W$.
So if the morphism from the non complete Jacobi algebra of $(Q,W)$
to the complete Jacobi algebra is bijective 
(which is not automatic even if the non complete
Jacobi algebra is finite-dimensional!), then
Bridgeland's construction applies to this case. In
particular, it often applies when the quiver admits a reddening sequence.
The problem of comparing Bridgeland's stability scattering diagram
with the cluster scattering diagram of 
Gross--Hacking--Keel--Kontsevich 
\cite{GrossHackingKeelKontsevich18} is the subject of
ongoing research \cite{CheungMandel19, Qin19}.
Other topics closely related to the representation-theoretic proof
are the investigation of chains of torsion classes induced by
stability conditions \cite{Igusa17a, Igusa17,BruestleSmithTreffinger17,
ApruzzeseIgusa18, BruestleSmithTreffinger18, Treffinger18},
the wall and chamber structure of the space of stability conditions
\cite{BruestleSmithTreffinger18a} and the study of
semi-invariants and picture groups \cite{IgusaOrrTodorovWeyman09,
IgusaOrrTodorovWeyman15, IgusaTodorovWeyman16}.

\appendix
\section{Maximal chains of torsion classes, by Laurent Demonet}

Let $A$ be a finite-dimensional algebra over a field $k$. We consider
the category $\mod A$ of  finite-dimensional right $A$-modules and its
lattice $\tors(A)=\tors(\mod A)$ of torsion classes (cf. section~\ref{s:proof}).
It is a {\em complete} lattice, \ie the associated category has all limits,
called meets, and all colimits, called joins. In particular, for each set
of modules $\cx\subseteq \mod A$, there is a smallest torsion class $\T(\cx)$ containing
$\cx$. An element $x$ of a complete poset is {\em completely join irreducible} if
it is not the join of an arbitrary family of elements $<x$. For a poset
$P$, we denote by $\Hasse(P)$ the {\em Hasse quiver} of $P$, whose
vertices are the elements of $P$ and which has an arrow $x \to y$ if
$x<y$ and $x=z$ or $z=y$ whenever $x\leq z\leq y$ (it is opposite
to the Hasse quiver of \cite{DemonetIyamaReadingReitenThomas17}).
If $P$ is complete, we denote by $\cjirr(P)$ the subposet of completely
join irreducible elements.

A {\em brick} is an $A$-module whose endomorphisms form a division algebra.
We write $\brick(A)$ for the set of isomorphism classes of bricks of $\mod A$.

A {\em chain of torsion classes} is a totally ordered subposet 
of $\tors(A)$. Chains of torsion classes are ordered by inclusion.
Using the results of \cite{DemonetIyamaReadingReitenThomas17}
we will describe the (possibly infinite!) maximal chains of torsion classes in
terms of the bricks of $\mod A$. Notice that part of the results
of \cite{DemonetIyamaReadingReitenThomas17} were independently
obtained in \cite{BarnardCarrollZhu17}.

Let $I$ be a totally ordered set. An {\em $I$-chain of bricks} is a map
\[
S_\bt: I \to \brick(A)\ko i \mapsto S_i
\]
such that $\Hom(S_i, S_j)=0$ for $i<j$. In particular, the map $S_\bt$
induces an injection from $I$ into the set of isomorphism classes of
bricks. If $I^1$ and $I^2$ are totally
ordered sets and $S_\bt^1$, $S_\bt^2$ are chains of bricks indexed
by $I^1$ and $I^2$, we write $S_\bt^1 \leq S_\bt^2$ if there is an increasing
inclusion $\iota: I^1 \to I^2$ such that for all $i\in I^1$, the module
$S^1_i$ is isomorphic to $S^2_{\iota(i)}$. Suppose we have an inequality
$S^1_\bt \leq S^2_\bt$ given by $\iota$ and an inequality 
$S^2_\bt \leq S^1_\bt$ given by $\kappa$.
Then for each $i\in I^1$, the module $S_i$ is isomorphic to $S_{\kappa(\iota(i))}$
so that $\kappa\iota(i)=i$ and $\iota$ and $\kappa$ are bijective. 
We define the chains of bricks $S_\bt^1$ and $S_\bt^2$ to be {\em equivalent}
if we have $S^1_\bt\leq S^2_\bt \leq S^1_\bt$.

Let $I$ be a totally ordered
set and $S_\bt$ an $I$-chain of bricks. We will associate with $S_\bt$
a chain of torsion classes $\ct_\bt=\Phi(S_\bt)$.
An {\em ideal} of $I$ is a subset
$j\subseteq I$ such that $l\in j$ whenever $l\leq i$ and $i\in j$.
Let $J$ be the poset of ideals of $I$. For each $j\in J$, let
$\ct_j$ be the smallest torsion class containing the $S_i$, $i\in j$. It is
the join of the torsion classes $\T(S_i)$, $i\in j$.
Clearly, the map $j \mapsto \ct_j$ is increasing. If we have $j\subset j'$ and
$i$ belongs to $j'$ but not to $j$, then the brick $S_i$ lies in $\ct_{j'}\cap \ct_j^\perp$
so that the map $j \mapsto \ct_j$ is strictly increasing. We define $\Phi(S_\bt)$
to be the subposet of $\tors(A)$ formed by the $\ct_j$, $j\in J$.

\begin{proposition} \label{prop1} The map $\Phi$ induces an injective morphism 
from the poset of equivalence classes of chains of bricks to the
poset of chains of torsion classes.
\end{proposition}

\begin{proof} Clearly, the map $\Phi$ is a morphism of posets. Let
$S_\bt: I \to \brick(A)$ and $S'_\bt: I'\to \brick(A)$ be chains of bricks
and $\ct_\bt=\Phi(S_\bt)$ and $\ct'_\bt=\Phi(S'_\bt)$ the associated 
chains of torsion classes. Let $J$ and $J'$ be the posets of ideals
of $I$ and $I'$. We first construct an isomorphism
of posets $\pi: I \iso I'$ such that for each $j\in J$, we have
$\ct'_{\pi(j)}=\ct_j$. For this, 
notice that $J$ and $\tors(A)$ are complete lattices and
that the injective morphism $j\mapsto \ct_j$ commutes with arbitrary joins.
We can recover $I$ from the set of ideals $J$ of $I$
as the subposet of the completely join irreducible
elements. Since we have isomorphisms of posets
\[
J \iso \{ \ct_j\;|\; j\in J\} = \{\ct'_{j'} \;|\; j'\in J'\} \liso J'
\]
we obtain an isomorphism $\pi: I \iso I'$ and it clearly satisfies
$\ct_j= \ct'_{\pi{j}}$ for all $j\in J$.
Thus, we may assume that $S_\bt$ and $S'_\bt$ are maps
$I \to \brick(A)$ such that for each ideal $j$ of $I$, the join 
$\ct_j$ of the classes $\T(S_i)$, $i\in j$, coincides with $\ct'_j$.
Fix an element $i$ of $I$. Let $j\subseteq I$ be the ideal of
the elements $i'<i$. Let $\cu=\ct_{j}=\ct'_{j}$. Let
$\ct$ be the torsion class $\cu\join \T(S_{i})=\cu \join \T(S'_{i})$.
Since $S_{i}$ belongs to $\cu^\perp$, by part a) of Theorem~3.4
of \cite{DemonetIyamaReadingReitenThomas17}, the torsion class
$\ct$ is completely join irreducible and the unique arrow of the
Hasse quiver $\Hasse [\cu,\mod A]$ ending at $\ct$ is labelled by $S_{i}$,
cf. Definition~3.5 of \cite{DemonetIyamaReadingReitenThomas17}.
Since we also have $\ct=\cu\join \T(S'_{i})$, the arrow is also
labeled by $S'_{i}$, which is therefore isomorphic to $S_{i}$.
\end{proof}

\begin{proposition} \label{prop2} Let $C$ be a maximal chain of torsion classes.
Then each $\ct\in C$ equals the join $\cv$ of the $\cu \in \cjirr(C)$ contained in $\ct$.
It also equals the join $\cv'$ of the classes $\T(S_q)$, where $q$ runs through the
arrows of the Hasse quiver of $C \cap [0,\ct]$.
\end{proposition}

\begin{proof} 
Clearly, we have $\cv\subseteq \ct$. Let
us show that $\cv'$ is contained in $\cv$. If $q: \cu' \to \cu$ is an arrow
of the Hasse quiver of $C\cap [0,\ct]$, it is also an arrow of $\Hasse(A)$,
by the maximality of $C$. Thus, the arrow has a well-defined brick label $S_q$.
Moreover, the brick $S_q$ belongs to $\cu$ and $\cu$ is completely join
irreducible. Therefore, we have $\T(S_q)\subset \cv$ and $\cv'\subset\cv$.
Let us show that the inclusion $\cv'\subseteq\ct$ is an equality.
By Lemma~3.10 of \cite{DemonetIyamaReadingReitenThomas17}, the
modules in $\ct\cap\cv'^\perp$ are those admitting a filtration whose subquotients
are bricks in $\ct\cap\cv'^\perp$. Let $S$ be a brick of minimal dimension in
$\ct\cap\cv'^\perp$. Consider the meet $\cw$ of the torsion classes $\cw'\in C$
containing $\cv'$ and $S$. Since $C$ is a maximal chain, it is stable under
meets and thus contains $\cw$. Let $\cv''\in C$ be a torsion class such that
$\cv'\subseteq \cv''\subseteq \cw$ and $\cv''\neq \cw$. Let $X\in \cv''$.
Consider a morphism $f: X \to S$. The image $\im(f)$ is a quotient of $X$
and thus belongs to $\ct$. It is also a submodule of $S$ and thus belongs to
$\cv'^\perp$. Thus, it belongs to $\cd\cap\cv'^\perp$ and has a filtration whose
subquotients are bricks in $\ct\cap\cv'^\perp$. By the minimality of the dimension
of $S$, we have $\im(f)=0$ or $\im(f)=S$. If we have $\im(f)=S$, then $S$
is a quotient of $X$ and belongs to $\cv''$, which contradicts the definition
of $\cw$. Thus, we have $\im(f)=0$ and $\cv''\subset \mbox{ }^\perp S\cap \cw$.
By the maximality of $C$, it follows that the class $\mbox{ }^\perp S\cap \cw$
belongs to $C$ and there is an arrow $\mbox{ }^\perp S\cap \cw \to \cs$
in $\Hasse(C)$ labeled by $S$, \cf Theorem~3.4 of 
\cite{DemonetIyamaReadingReitenThomas17}. So the module $S$ belongs
to $\cv$, which is a contradiction.
\end{proof}

\begin{theorem} \label{thm:chains} The map $\Phi: S_\bt \mapsto \ct_\bt$
induces a bijection from the set of equivalence classes of maximal chains
of bricks to the set of maximal chains of torsion classes.
\end{theorem}

\begin{proof} By Proposition~\ref{prop1}, it only remains to prove that $\Phi$ is
surjective. Let $C$ be a maximal chain of torsion classes. Let $I$ be the
poset $\cjirr(C)$. For each $i\in I$, there is a unique arrow $\cu_i \to i$
in $\Hasse(C)$. By the maximality, it is also an arrow of $\Hasse(A)$.
Let $S_i$ be its label, \cf Definition~3.5 of \cite{DemonetIyamaReadingReitenThomas17}.
We claim that the map $i \mapsto S_i$ is an $I$-chain of bricks.
Indeed, if we have $i<j$ in $I$, then $i\subseteq \cu_j \subseteq j$ and so
$S_j\in \cu_j^\perp$ and $S_i\in i\subseteq \cu_j$ so that we have
$\Hom(S_i,S_j)=0$. Let $\ct_\bt$ be the chain of torsion classes
$\Phi(S_\bt)$. For $\cu\in C$, consider the ideal $j$ of $I$ formed by the
$i\in I$ contained in $U$. By Proposition~\ref{prop2}, we have $\cu=\ct_j$.
Whence an inclusion $C \subseteq \ct_\bt$. By the maximality of $C$,
we actually have an equality.
\end{proof}



\def\cprime{$'$} \def\cprime{$'$}
\providecommand{\bysame}{\leavevmode\hbox to3em{\hrulefill}\thinspace}
\providecommand{\MR}{\relax\ifhmode\unskip\space\fi MR }
\providecommand{\MRhref}[2]{%
  \href{http://www.ams.org/mathscinet-getitem?mr=#1}{#2}
}
\providecommand{\href}[2]{#2}

\end{document}